\documentclass[12pt,reqno]{amsart}
\usepackage{times}

\newtheorem{thm}{Theorem}
\newtheorem{lemma}[thm]{Lemma}
\newtheorem{cor}[thm]{Corollary}

\newtheorem*{mthm}{Main Theorem}

\newcommand{\C}{{\mathbb C}}
\newcommand{\R}{{\mathbb R}}
\newcommand{\cn}{{\mathbb C}^n}

\newcommand{\dist}{{\rm dist}\,}
\newcommand{\Dom}{{\rm Dom}\,}

\begin{document}

\title[Uncertainty Principle]{Uncertainty Principles for the Fock Space}

\author{Kehe Zhu}
\address{Department of Mathematics and Statistics\\
         State University of New York\\
         Albany, NY 12222, USA}
\email{kzhu@math.albany.edu}


\begin{abstract}
We prove several versions of the uncertainty principle for the Fock space $F^2$ in the complex plane. 
In particular, for any unit vector $f$ in $F^2$, we show that
$$\dist(f'+zf,[f])\,\dist(f'-zf,[f])\ge1,$$
where $[f]=\C f$ is the one-dimensional subspace spanned by $f$. We also determine exactly when 
equality occurs above.
\end{abstract}

\maketitle

\section{Introduction}

Heisenberg's Uncertainty Principle in quantum physics states that the position and velocity (or momentum) of a particle 
cannot be measured exactly at the same time. This has nothing to do with the possible imperfection of the measuring 
instrument. It is simply not possible, in theory, to achieve exact measurements for both at the same time. More specifically, 
the product of the ``uncertainty" for the position and the ``uncertainty" for the velocity (or momentum) of a particle is always 
greater than or equal to a tiny positive constant, namely, $h/(4\pi)$, where $h$ is Planck's constant.

There exist many similar uncertainty principles, or sometimes called indeterminacy principles, in physics,
engineering, and mathematics, that are based on position, velocity, momentum, energy, time, and so on. 
In this paper we are going to obtain several versions of the uncertainty principle in the context of the Fock space.
Not surprisingly, the Fock space is one of the significant mathematical tools used in quantum physics.

Let $\C$ be the complex plane $\C$ and
$$d\lambda(z)=\frac1\pi e^{-|z|^2}\,dA(z)$$
be the Gaussian measure on $\C$, where $dA$ is ordinary area measure. We define the Fock space as follows: 
$$F^2=L^2(\C,d\lambda)\cap H(\C),$$
where $H(\C)$ is the space of all entire functions. See \cite{Z} for general properties of the Fock space.

The purpose of the paper is to obtain several versions of the uncertainty principle for the Fock space.
We summarize our results as follows.

\begin{mthm}
For any function $f\in F^2$ and any real numbers $a$ and $b$ we have
$$\|f'+zf-af\|\|f'-zf-ibf\|\ge\|f\|^2.$$
As a consequence, if $f$ is a unit vector in $F^2$, we also have
\begin{enumerate}
\item[(a)] $\dist(f'+zf,[f])\,\dist(f'-zf,[f])\ge1$.
\smallskip
\item[(b)] $\|f'+zf\|\|f'-zf\||\sin(\theta_+)\sin(\theta_-)|\ge1$.
\smallskip
\item[(c)] $(\|f'\|^2+\|zf\|^2)|\sin(\theta_+)\sin(\theta_-)|\ge1$.
\end{enumerate}
Here $[f]$ is the one-dimensional subspace in $F^2$ spanned by $f$, and $\theta_{\pm}$ are the angles
between $f$ and $f'\pm zf$. We also determine exactly when equality holds in each of the cases above.
\end{mthm}

Note that the interesting case here is when the function $f'(z)$ (or equivalently, the function $zf(z)$) also belongs 
to the Fock space $F^2$, which is not always the case. See \cite{CZ}. When the function $f'(z)$ is not in $F^2$, 
each of the left-hand sides of the inequalities above is infinite, so the inequality becomes trivial.

I wish to thank Hans Feichtinger and Bruno Torresani for inviting me to CIRM at Luminy, Marseille, in the fall 2014 
semester. They introduced me to the wonderful area of Gabor Analysis and this research was initiated while I
was visiting them.

\section{Uncertainty principles for the Fock space}

Uncertainty principles are usually consequences of the following general result from functional analysis.
There is no exception here.

\begin{thm}
Suppose $A$ and $B$ are self-adjoint operators, possibly unbounded, on a Hilbert space $H$. Then
\begin{equation}
\|(A-a)x\|\|(B-b)x\|\ge\frac12|\langle[A,B]x,x\rangle|
\label{eq1}
\end{equation}
for all $x\in\Dom(AB)\cap\Dom(BA)$ and all $a,b\in\R$. Here 
$$[A,B]=AB-BA$$
is the commutator of $A$ and $B$. Furthermore, equality in (\ref{eq1}) holds if and only if $(A-a)x$
and $(B-b)x$ are purely imaginary scalar multiples of one another.
\label{1}
\end{thm}

\begin{proof}
This is well known. See page 27 of \cite{F, G}.
\end{proof}

Every time there are natural self-adjoint operators $A$ and $B$ such that $[A,B]$ is a scalar multiple of
the identity operator, an uncertainty principle arises. We will construct two such operators on $F^2$
based on the operator of multiplication by $z$ and the operator of differentiation.

\begin{lemma}
Let $D:F^2\to F^2$ be the operator of differentiation, that is, $Df=f'$. Then its adjoint $D^*$ is given
by $(D^*f)(z)=zf(z)$.
\label{2}
\end{lemma}

\begin{proof}
The standard orthonormal basis of $F^2$ is given by
$$e_n(z)=\frac{z^n}{\sqrt{n!}},\qquad n=0,1,2,\cdots.$$
If
$$f=\sum_{n=0}^\infty a_ne_n,\qquad g=\sum_{n=0}^\infty b_ne_n,$$
are polynomials (which are dense in $F^2$), then
$$Df(z)=\frac d{dz}\sum_{n=0}^\infty\frac{a_n}{\sqrt{n!}}\,z^n=\sum_{n=1}^\infty a_n\frac n{\sqrt{n!}} z^{n-1}
=\sum_{n=0}^\infty a_{n+1}\sqrt{n+1}e_n(z).$$
Also,
$$zg(z)=\sum_{n=0}^\infty\frac{b_n}{\sqrt{n!}}\,z^{n+1}=\sum_{n=1}^\infty\sqrt nb_{n-1}e_n(z).$$
It follows that
$$\langle Df,g\rangle=\sum_{n=0}^\infty\sqrt{n+1}a_{n+1}\overline b_n
=\sum_{n=1}^\infty\sqrt na_n\overline b_{n-1}=\langle f,zg\rangle.$$
This proves the desired result.
\end{proof}

It is easy to check that $[D,D^*]=I$. In fact,
\begin{equation}
(DD^*-D^*D)f=(zf)'-zf'=f
\label{eq2}
\end{equation}
for all $f\in F^2$. In the quantum theory of harmonic oscillators, such a $D$ is called an annihilation operator
and $D^*$ is called a creation operator. However, we cannot apply Theorem~\ref{1} directly to the standard 
commutation relation $[D,D^*]=I$, because the operators $D$ and $D^*$ are not self-adjoint.
 
Thus we consider the following two self-adjoint operators on $F^2$:
$$A=D+D^*,\qquad B=i(D-D^*).$$
More specifically, we have
\begin{equation}
Af(z)=f'(z)+zf(z),\qquad Bf(z)=i(f'(z)-zf(z)).
\label{eq3}
\end{equation}
It follows from \cite{CZ} that, for a function $f\in F^2$, we have $f'\in F^2$ if and only if $zf\in F^2$.
If $f'\in F^2$, then both $Af$ and $Bf$ are well defined. If both $f'+zf$ and $f'-zf$ are in $F^2$,
then by taking the sum and difference of these two functions, we see that $f'$ and $zf$ are both
in $F^2$. Therefore, the intersection of the domains of $A$ and $B$ consists of those functions
$f$ such that $f'$ (or $zf$) is still in $F^2$. It is possible to identify the domains of $AB$, $BA$, and
their intersection as well.

\begin{lemma}
For the operators $A$ and $B$ defined above, we have $[A,B]=-2iI$, where $I$ is the identity
operator on $F^2$ and $i$ is the imaginary unit.
\label{3}
\end{lemma}

\begin{proof}
We have
\begin{eqnarray*}
AB-BA&=&i\left[(D+D^*)(D-D^*)-(D-D^*)(D+D^*)\right]\\
&=&2i[D^*D-DD^*]=-2iI.
\end{eqnarray*}
The last equality above follows from (\ref{eq2}).
\end{proof}

We now derive the first version of the uncertainty principle for the Fock space.

\begin{thm}
Let $f\in F^2$. We have
\begin{equation}
\|f'+zf-af\|\|f'-zf-ibf\|\ge\|f\|^2
\label{eq4}
\end{equation}
for all real numbers $a$ and $b$. Furthermore, equality holds if and only if there exist a positive constant $c$
and a complex constant $C$ such that
$$f(z)=C\exp\left(\frac{c-1}{2(c+1)}z^2+\frac{a+ibc}{c+1}z\right).$$
\label{4}
\end{thm}

\begin{proof}
The inequality in (\ref{eq4}) follows from Theorem~\ref{1} and Lemma~\ref{3}, because $a$ and $b$ are arbitrary and
$$\|i(f'-zf)+bf\|=\|f'-zf-ibf\|.$$ 
Also, it follows from Theorem~\ref{1} that equality holds in (\ref{eq4}) if and only if there exists a real constant $c$ such that
\begin{equation}
f'+zf-af=ic\left[i(f'-zf)+bf\right]=-c(f'-zf)+ibcf,
\label{eq5}
\end{equation}
which can be rewritten as
\begin{equation}
(1+c)f'+\left[(1-c)z-(a+ibc)\right]f=0.
\label{eq6}
\end{equation}
If $c=-1$, the only solution for (\ref{eq6}) is $f=0$. If $c\not=-1$, it is elementary to show that the general solution 
of (\ref{eq6}) is given by
\begin{equation}
f(z)=C\exp\left(\frac{c-1}{2(c+1)}z^2+\frac{a+ibc}{c+1}z\right),
\label{eq7}
\end{equation}
where $C$ is any complex constant. 

It is well known that every function $f\in F^2$ satisfies
\begin{equation}
\lim_{z\to\infty}f(z)\exp\left(-\frac12|z|^2\right)=0.
\label{eq8}
\end{equation}
See page 38 of \cite{Z} for example. Therefore, a necessary condition for the function in (\ref{eq7})
to be in $F^2$ is that either $C=0$ or $|c-1|\le|c+1|$. Since $c$ is real, this is equivalent to $(c-1)^2\le(c+1)^2$,
or $c\ge0$. When $c=0$, the function in (\ref{7}) becomes
$$f(z)=C\exp\left(-\frac12z^2+az\right).$$
Choosing $z=ir$, where $r\to+\infty$, and applying (\ref{eq8}), we conclude that $C$ must be zero.
This completes the proof of the theorem.
\end{proof}

To derive the next version of the uncertainty principle for the Fock space, we now fix some function
$f\in F^2$. Since $A$ is self-adjoint, for any real $a$ we have
\begin{eqnarray*}
\|(A-a)f\|^2&=&\|Af\|^2+|a|^2\|f\|^2-2a\langle Af,f\rangle\\
&=&\|Af\|^2+\|f\|^2\left[\left|a-\frac{\langle Af,f\rangle}{\|f\|^2}\right|^2-\frac{|\langle Af,f\rangle|^2}{\|f\|^4}\right]\\
&\ge&\|Af\|^2-\frac{|\langle Af,f\rangle|^2}{\|f\|^2}.
\end{eqnarray*}
This shows that
$$\min_{a\in\R}\|(A-a)f\|^2=\|Af\|^2-\frac{|\langle Af,f\rangle|^2}{\|f\|^2},$$
and the minimum is attained when
$$a=\frac{\langle Af,f\rangle}{\|f\|^2}.$$
In other words, we have
$$\min_{a\in\R}\|f'+zf-af\|^2=\|f'+zf\|^2-\frac{|\langle f'+zf,f\rangle|^2}{\|f\|^2},$$
and the minimum is attained when
$$a=\frac{\langle f'+zf,f\rangle}{\|f\|^2}.$$

Similarly, we have
\begin{eqnarray*}
\min_{b\in\R}\|f'-zf-ibf\|^2&=&\min_{b\in\R}\|i(f'-zf)+bf\|^2\\
&=&\|f'-zf\|^2-\frac{|\langle f'-zf,f\rangle|^2}{\|f\|^2},
\end{eqnarray*}
and the minimum is attained when 
$$b=-\frac{\langle f'-zf,f\rangle}{\|f\|^2}\,i.$$
Specializing to the case when $f$ is a unit vector, we obtain the following version of the uncertainty principle.

\begin{cor}
If $f$ is a unit vector in $F^2$, then
$$(\|f'+zf\|^2-|\langle f'+zf,f\rangle|^2)(\|f'-zf\|^2-|\langle f'-zf,f\rangle|^2)\ge1.$$
Furthermore, equality holds if and only if 
$$f(z)=C\exp\left(\frac{c-1}{2(c+1)}\,z^2+\frac{a+ibc}{c+1}\,z\right),$$
where $c$ is positive, $a$ and $b$ are real, $C$ is complex, and these constants satisfy $\|f\|=1$.
\label{5}
\end{cor}

\begin{proof}
The desired inequality follows from Theorem~\ref{4} and the minimization argument following it. 
Also, by Theorem~\ref{4} and the minimization argument, we see that we can achieve equality 
if and only if (\ref{eq5}) is satisfied for some positive $c$ and for
$$a=\langle f'+zf,f\rangle,\qquad b=-i\langle f'-zf,f\rangle,$$
that is,
\begin{equation}
f'+zf-\langle f'+zf,f\rangle f=-c\left[f'-zf-\langle f'-zf,f\rangle f\right].
\label{eq9}
\end{equation}

If $f(z)$ is a function in the specified form, then by Theorem~\ref{4}, we have
$$\|f'+zf-af\|^2\|f'-zf-ibf\|^2=1,$$
which together with the minimization argument forces
$$(\|f'+zf\|^2-|\langle f'+zf,f\rangle|^2)(\|f'-zf\|^2-|\langle f'-zf,f\rangle|^2)=1.$$
On the other hand, if this equality holds, then (\ref{eq9}) holds, which means (\ref{eq6}) holds with
$a=\langle f'+zf,f\rangle$ and $b=\langle f'-zf,f\rangle$. Solving equation (\ref{eq6}) leads to the
specified form for $f$. The completes the proof of the corollary.
\end{proof}

\begin{cor}
For any $f\in F^2$ we have
$$\|f'+zf\|\|f'-zf\|\ge\|f\|^2.$$
Furthermore, equality holds if and only if
$$f(z)=C\exp\left(\frac{c-1}{2(c+1)}\,z^2\right)$$
for some positive $c$ and some complex $C$, or equivalently, $f(z)=Ce^{rz^2}$, where $C$ is complex and
$-\frac12<r<\frac12$.
\label{6}
\end{cor}

\begin{proof}
This follows directly from Theorem~\ref{4} by setting $a=b=0$.
\end{proof}

Since
\begin{eqnarray*}
\|f'\|^2+\|zf\|^2&=&\frac12\left[\|f'+zf\|^2+\|f'-zf\|^2\right]\\
&\ge&\|f'+zf\|\|f'-zf\|,
\end{eqnarray*}
it follows from Corollary~\ref{6} that 
$$\|f'\|^2+\|zf\|^2\ge\|f\|^2,\qquad f\in F^2.$$
This is a trivial inequality though, as can easily be seen from the Taylor expansion of $f$ and the standard orthonormal 
basis for $F^2$. However, we can modify the argument above to obtain something nontrivial and more interesting.

\begin{cor}
For any $f\in F^2$ and $\sigma>0$ we have
$$\frac\sigma2\|f'+zf\|^2+\frac1{2\sigma}\|f'-zf\|^2\ge\|f\|^2.$$
Moreover, equality holds if and only if
$$f(z)=C\exp\left(\frac{1-\sigma}{2(1+\sigma)}\,z^2\right),$$
where $C$ is any complex constant.
\label{7}
\end{cor}

\begin{proof}
By Corollary \ref{6}, we have
\begin{eqnarray*}
\|f\|^2&\le&\|\sqrt\sigma(f'+zf)\|\|(f'-zf)/\sqrt\sigma\|\\
&\le&\frac12\left[\|\sqrt\sigma(f'+zf)\|^2+\|(f'-zf)/\sqrt\sigma\|^2\right]\\
&=&\frac\sigma2\|f'+zf\|^2+\frac1{2\sigma}\|f'-zf\|^2.
\end{eqnarray*}
Furthermore, equality holds if and only if
$$f'+zf=-c(f'-zf),\qquad \|\sqrt\sigma(f'+zf)\|=\|(f'-zf)/\sqrt\sigma\|.$$
This is equivalent to
$$f(z)=C\exp\left(\frac{c-1}{2(c+1)}\,z^2\right),\qquad c\sigma=1.$$
The proof is complete.
\end{proof}

We will comment more on this version of the uncertainty principle in the next section. We now obtain several
versions of the uncertainty principle that are based on the geometric notions of 
angle and distance. These results provide better estimates than those in Corollaries \ref{5} to \ref{7}.

\begin{cor}
Suppose $f$ is any function in $F^2$, not identically zero, and $\theta_{\pm}$ are the angles between $f$ and 
$f'\pm zf$. Then
$$\|f'+zf\|\|f'-zf\||\sin(\theta_+)\sin(\theta_-)|\ge\|f\|^2.$$
Furthermore, equality holds if and only if
$$f(z)=C\exp\left(\frac{c-1}{2(c+1)}\,z^2+\frac{a+ibc}{c+1}\,z\right),$$
where $c$ is positive, $a$ and $b$ are real, and $C$ is complex and nonzero.
\label{8}
\end{cor}

\begin{proof}
We actually have
\begin{eqnarray*}
\|f'+zf\|^2-\frac{|\langle f'+zf,f\rangle|^2}{\|f\|^2}&=&\|f'+zf\|^2\left[1-\left|\frac{\langle f'+zf,f\rangle}
{\|f'+zf\|\|f\|}\right|^2\right]\\
&=&\|f'+zf\|^2(1-\cos^2(\theta_+))\\
&=&\|f'+zf\|^2\sin^2(\theta_+).
\end{eqnarray*}
The same argument shows that
$$\|f'-zf\|^2-\frac{|\langle f'-zf,f\rangle|^2}{\|f\|^2}=\|f'-zf\|^2\sin^2(\theta_-).$$
The desired result then follows from Corollary~\ref{5} and its proof.
\end{proof}

\begin{cor}
Suppose $f$ is a unit vector in $F^2$ and $\theta_{\pm}$ are the angles between $f$ and $f'\pm zf$ in $F^2$.
Then for any $\sigma>0$ we have
$$\left(\frac\sigma2\|f'+zf\|^2+\frac1{2\sigma}\|f'-zf\|^2\right)|\sin(\theta_+)\sin(\theta_-)|\ge1,$$
with equality if and only if $f$ satisfies the integral-differential equation (\ref{eq9}) and
$\|f'-zf\|=\sigma\|f'+zf\|$. In particular,
$$(\|f'\|^2+\|zf\|^2)|\sin(\theta_+)\sin(\theta_-)|\ge1,$$
with equality if and only if $f$ satisfies the integral-differential equation (\ref{eq9}) and the identity
$\|f'+zf\|=\|f'-zf\|$.
\label{9}
\end{cor}

\begin{proof}
This follows from Corollaries \ref{7} and \ref{8} and their proofs.
\end{proof}

\begin{cor}
Suppose $f$ is any function in $F^2$, not identically zero. Then
$$\dist(f'-zf,[f])\,\dist(f'+zf,[f])\ge\|f\|^2,$$
where $[f]$ is the one-dimensional subspace of $F^2$ spanned by $f$ and $d(g,X)$ denotes the distance
in $F^2$ from $g$ to $X$. Furthermore, equality holds if and only if
$$f(z)=C\exp\left(\frac{c-1}{2(c+1)}\,z^2+\frac{a+ibc}{c+1}\,z\right),$$
where $c$ is positive, $a$ and $b$ are real, and $C$ is complex and nonzero.
\label{10}
\end{cor}

\begin{proof}
This is a restatement of Corollary~\ref{8}, because
$$\dist(f'+zf,[f])=\|f'+zf\||\sin(\theta_+)|,$$
and
$$\dist(f'-zf,[f])=\|f'-zf\||\sin(\theta_-)|,$$
which are clear from elementary observations.
\end{proof}

\section{Some generalizations}

Traditionally, uncertainty principles are stated in terms of real parameters $a$ and $b$. However, Theorem~\ref{1}
can easily be extended to the case of complex parameters.

\begin{thm}
Suppose $A$ and $B$ are self-adjoint operators, possibly unbounded, on a Hilbert space $H$. Then
$$\|(A-a)x\|\|(B-b)x\|\ge\frac12|\langle[A,B]x,x\rangle|$$
for all $x\in\Dom(AB)\cap\Dom(BA)$ and all complex numbers $a$ and $b$. Moreover, equality holds if and only 
if both $a$ and $b$ are real, and the vectors $(A-a)x$ and $(B-b)x$ are purely imaginary multiples of one another.
\end{thm}

\begin{proof}
Write $a=a_1+ia_2$, where both $a_1$ and $a_2$ are real. Then
\begin{eqnarray*}
\|(A-a)x\|^2&=&\|(A-a_1)x-ia_2x\|^2\\
&=&\|(A-a_1)x\|^2+|a_2|^2\|x\|^2\\
&&\ -ia_2\langle x,(A-a_1)x\rangle+a_2i\langle(A-a_1)x,x\rangle\\
&=&\|(A-a_1)x\|^2+|a_2|^2\|x\|^2\ge\|(A-a_1)x\|^2.
\end{eqnarray*}
Similarly,
$$\|(B-b)x\|^2=\|(B-b_1)x\|^2+|b_2|^2\|x\|^2\ge\|(B-b_1)x\|^2$$
if $b=b_1+ib_2$. The desired result then follows from Theorem~\ref{1} and its proof.
\end{proof}

Another extension we want to mention here is that everything we have done for the Fock space remains
valid for any annihilation operator $D$ and its associated creation operator $D^*$.

\begin{thm}
Suppose $D$ is any operator on $H$ such that $[D,D^*]=I$. Then
$$\|Dx+D^*x-ax\|\|Dx-D^*x-ibx\|\ge\|x\|^2$$
for all $x\in\Dom(DD^*)\cap\Dom(D^*D)$ and all real numbers $a$ and $b$. Equality holds if and only if 
there exists a real constant $c$ such that
$$Dx+D^*x-ax=-c(D^*x-Dx)+ibcx.$$
\end{thm}

\begin{proof}
This follows from the proofs of Lemma~\ref{3} and Theorem~\ref{4}.
\end{proof}

Again, it is easy to extend the result to the case of complex parameters, and just like the case of the Fock
space, several different versions of the uncertainty principle can be formulated for such a pair of annihilation
and creation operators. Details are left to the interested reader.

\section{Further remarks}

As was mentioned at the beginning of Section 2, we have an uncertainty principle whenever
we have two self-adjoint operators whose commutator is a constant multiple of the identity. In the case
of the Fock space, the operators $Af=f'+zf$ and $Bf=i(f'-zf)$ are easily shown to be self-adjoint and
they have a nice commutation relation. However, our motivation actually came from classical Fourier analysis.

More specifically, the operator $X$ of multiplication by $x$ on $L^2(\R)=L^2(\R,dx)$ is clearly self-adjoint.
It is also well known that the operator 
$$D=\frac1{2\pi i}\frac{d}{dx}$$ 
is self-adjoint on $L^2(\R)$. Via the so-called Bargmann transform (see \cite{F,G,Z}), which is an isometric linear 
transformation from $L^2(\R)$ onto the Fock space $F^2$, it can be shown that $X$ is unitarily equivalent to the 
operator $A/2$, and $D$ is unitarily equivalent to $-B/(2\pi)$, where $A$ and $B$ are defined in (\ref{eq3}). The 
well-known commutation relation 
$$[X,D]=-\frac1{2\pi i}I$$ on $L^2(\R)$ then translates into the commutation relation $[A,B]=-2i I$ on $F^2$.

The classical uncertainty principle in Fourier analysis states that
$$\|Xf\|^2+\|Df\|^2\ge\frac1{2\pi}\|f\|^2$$
for all $f\in L^2(\R)$, with equality holding if and only if 
$$f(x)=C\exp(-\pi x^2)$$ 
is a multiple of the standard Gaussian. See Corollary 2.2.3 of \cite{G} and Corollary 1.37 of \cite{F}. Via the Bargmann 
transform we see that this classical uncertainty principle is equivalent to the following inequality on the Fock space:
$$\frac14\|f'+zf\|^2+\frac1{4\pi^2}\|f'-zf\|^2\ge\frac1{2\pi}\|f\|^2.$$
This is a consequence of our Corollary~\ref{7} with the choice of $\sigma=\pi$.

Alternatively, we could have obtained all our results for the Fock space from the classical
uncertainty principle in Fourier analysis via the Bargmann transform. We chose to do things directly in order to
avoid unnecessary background material that are not familiar to many complex analysts.

The results in the paper can be generalized to the case of Fock spaces in $\cn$ with a weighted Gaussian measure
$$d\lambda_\alpha(z)=\left(\frac\alpha\pi\right)^n e^{-\alpha|z|^2}\,dv(z).$$
We leave the details to the interested reader.

Finally, we remark that it is now natural to raise the problem of establishing uncertainty principles for other familiar 
function spaces such as the Hardy space and the Bergman space. Equivalently, for such spaces, we want to
construct natural self-adjoint operators whose commutator is a multiple of the identity operator. We intend to
continue this line of research in a subsequent paper.


\begin{thebibliography}{99}

\bibitem{CZ} H. Cho and K. Zhu, Fock-Sobolev spaces and their Carleson measures, {\it J. Funct. Anal.} {\bf 263} (2012), 2483-2506, 

\bibitem{F} G. Folland, {\it Harmonic Analysis on Phase Space}, Annals of Mathematics Studies {\bf 122}, Princeton University
Press, Princeton, New Jersey, 1989.

\bibitem{G} K. Gr\"ochenig, {\it Foundations of Time-Frequency Analysis}, Birkh\"auser, Boston, 2001. 

\bibitem{Z} K. Zhu, {\it Analysis on Fock Spaces}, Springer-Verlag, New York, 2012.

\end{thebibliography}
\end{document}